\documentclass[reqno, 11pt]{amsart}%
\usepackage{amssymb}
\usepackage[nesting]{hyperref}
\usepackage[pdftex]{graphicx}
\usepackage{listings}
\usepackage{multirow}
\usepackage{placeins}
\usepackage{color}
\usepackage{subfigure}
\usepackage{lscape}
\usepackage{dsfont}
\usepackage{amsmath}
\usepackage{amsfonts}%
\usepackage{tikz}
\usepackage{verbatim}
\providecommand{\U}[1]{\protect\rule{.1in}{.1in}}
\newcommand{\BlackBoxes}{\global\overfullrule5pt}

\BlackBoxes

\newcommand{\R}{\mathbb{R}}
\newcommand{\N}{\mathbb{N}}

\newcommand{\Eop}{\mathbb{E}}

\newcommand{\Q}{\mathbb{Q}}

\newcommand{\Supp}{\mathrm{supp}}

\newtheorem{theorem}{Theorem}
\newtheorem{corollary}[theorem]{Corollary}
\newtheorem{lemma}[theorem]{Lemma}

\theoremstyle{definition}

\newtheorem{remark}[theorem]{Remark}
\newtheorem{definition}[theorem]{Definition}
\numberwithin{equation}{section} \numberwithin{theorem}{section}
\def\0{\kern0pt\-\nobreak\hskip0pt\relax}

\makeatletter
\AtBeginDocument{ \def\@serieslogo{ \vbox to\headheight{ \parindent\z@ \fontsize{6}{7\p@}\selectfont
\vss}}}

\def\makeoverbar#1#2#3#4#5#6#7{ \setbox0=\hbox{$\m@th#2\mkern#5mu{{}#3{}}\mkern#6mu$} \setbox1=\null \dimen@=#4\fontdimen8#13 \dimen@=3.5\dimen@
\advance\dimen@ by \ht0 \dimen@=-#7\dimen@ \advance\dimen@ by \wd0
\ht1=\ht0 \dp1=\dp0 \wd1=\dimen@
\dimen@=\fontdimen8#13 \fontdimen8#13=#4\fontdimen8#13
\rlap{\hbox to \wd0{$\m@th\hss#2{\overline{\box1}}\mkern#5mu$}}
\fontdimen8#13=\dimen@}
\def\mylabel#1#2{{\def\@currentlabel{#2}\label{#1}}}
\makeatother

\begin{document}
\title[Martingale Optimal Transport in the Discrete Case Via Simple LP Techniques]{Martingale Optimal Transport in the Discrete Case Via Simple Linear Programming Techniques}
\author[N. \smash{B\"auerle}]{Nicole B\"auerle${}^*$}
\address[N. B\"auerle]{Department of Mathematics,
Karlsruhe Institute of Technology (KIT), D-76128 Karlsruhe, Germany}

\email{\href{mailto:nicole.baeuerle@kit.edu}
{nicole.baeuerle@kit.edu}}

\author[D. \smash{Schmithals}]{Daniel Schmithals${}^*$}
\address[D. Schmithals]{Department of Mathematics,
Karlsruhe Institute of Technology (KIT), D-76128 Karlsruhe, Germany}

\email{\href{daniel.schmithals@kit.edu} {daniel.schmithals@kit.edu}}

\thanks{${}^*$ Department of Mathematics,
Karlsruhe Institute of Technology (KIT), D-76128 Karlsruhe, Germany}
\begin{abstract}
We consider the problem of finding consistent upper price bounds and super replication strategies for exotic options, given the observation of call prices in the market. This field of research is called model-independent finance and has been introduced by \cite{hobson1998robust}. Here we use the link to mass transport problems. In contrast to existing literature we assume that the marginal distributions at the two time points we consider are discrete probability distributions. This has the advantage that the optimization problems reduce to linear programs and can be solved rather easily when assuming a general martingale Spence Mirrlees condition. We will prove the optimality of left-monotone transport plans under this assumption and provide an algorithm for its construction. Our proofs are simple and do not require much knowledge of probability theory. At the end we present an example to illustrate our approach.

\end{abstract}
\maketitle


\makeatletter \providecommand\@dotsep{5} \makeatother



\vspace{0.5cm}
\begin{minipage}{14cm}
{\small
\begin{description}
\item[\rm \textsc{ Key words} ]
{\small martingale optimal transport; linear programming; convex order; left-monotone transport plan}
\end{description}
}
\end{minipage}

\section{Introduction}
Classical models in financial mathematics work similar in principle. They fix a certain underlying probability space and assume that the random future behaviour of the underlying asset price process is specified somehow, for example as the solution of a stochastic differential equation. Further assuming no-arbitrage and completeness of the considered financial market, a unique equivalent martingale measure, i.e. a probability measure such that the discounted asset price process is a martingale, exists by the fundamental theorem of asset pricing. Then using the law of one price, it is possible to derive the uniquely determined price of an exotic option written on the underlying by calculating either the expected discounted payoff of the exotic option with respect to the equivalent martingale measure or the price of a self-financing, replicating hedging strategy. Most of the models are of parametric form, where the parameters are determined by calibrating the model to observable market prices of certain options.

The assumptions which are made are simplifying and often quite  unrealistic thus leading to model prices which are likely to be inaccurate and unreliable. Indeed, various studies, see e.g. \cite{schoutens2003perfect}, observe a great range of option prices when calibrating several different models to the same underlying market.

On the other hand, over the years, trading of call options became so liquid that  \cite{dupire1993model, dupire1994pricing} and others argued that they should rather be considered as contingent claims with exogenously fixed prices. Thus, the prices of European call options became available as information for pricing other, more complicated exotic options. In model-independent finance, this is used under the idealizing assumption that for the maturities of interest, the call option prices are observable for a continuum of strike prices.

The target of model-independent finance is now to price exotic options such that the prices satisfy no-arbitrage and are consistent with  observable call option prices. Therefore, no specific martingale measure but a set of different consistent martingale measures emerges from the analysis. Thus, no unique option price but a range of possible option prices may be derived. In return, the risk of model misspecification is eliminated. This intuition was first formalized by Hobson in his famous paper \cite{hobson1998robust}, where he used the no-arbitrage assumption and the knowledge about the call option prices to derive upper and lower price bounds for the lookback option in continuous time.

Such price bounds can be derived in various ways. There is one stream of literature which uses Skorokhod-type stopping problems to derive upper bounds (see e.g. \cite{hobson1998robust}) and another one using methods from optimal transport, see e.g. \cite{beiglbock2013model}. We will pursue the latter approach in this paper.

 In optimal transport, the problem is to minimize the cost that transportation of mass from one point to another generates in the sense that a cost-minimal transport allocation is aimed for. Mathematically, we may specify the mass at the origins and the destinations by measures. Then, minimizing the transport cost is equivalent to minimizing the integral over a usually bivariate function representing the cost of transporting a unit of mass from one point to another with respect to the set of all couplings or so-called transport plans which have the specified measures as marginals. The problem was originally introduced by Monge \cite{monge1781memoire} in 1781 and then refined by Kantorovich \cite{kantorovich1948problem, kantorovich2006problem} in 1948. A great variety of researchers considered the optimal transport problem and in the course of their research many important results on optimal transport were established, see for example  \cite{rachev1998mass, rachev1998application} or  \cite{villani2008optimal} for excellent monographes on the topic.

  Observing that there is an analogy between model-independent finance and optimal transport, as in both areas the marginals of the distribution over which some function is optimized are specified, \cite{beiglbock2013model} introduced a new research field that we refer to as martingale optimal transport. Reinterpreting the transport cost function as the payoff function of an exotic option and implementing the usual martingale condition of mathematical finance, the minimization problem of optimal transport cost evolves to the lower price bound problem of model-independent finance. Also, properly implementing the martingale condition in the dual problem of optimal transport, a pricing-hedging duality is shown using only the usual assumptions of model-independent finance, no-arbitrage and consistency with call options prices.

In this paper we consider the martingale optimal transport with discrete margins. This is in contrast to previous papers, see e.g. \cite{henry2016explicit} where continuous margins are  assumed. Moreover, in the discrete case it is possible to solve the problem which reduces to a linear program  with simple techniques. We do not rely on any previous result in this direction and present a stand-alone work. We are also able to generalize some notions like the martingale Spence Mirrlees condition which appears in \cite{henry2016explicit}. For payoff functions with this property we prove the optimality of a left-monotone transport plan, show its uniqueness and present an algorithm to compute the optimal solution. Under the martingale Spence Mirrlees condition the optimality of a left-monotone transport plan has for continuous marginals been shown in \cite{henry2016explicit} and for general marginals in \cite{beiglbock2017monotone}. Also note that the recent paper \cite{hobno18} considers the construction of optimal left-monotone transport plans for general margins. However, their results are far less explicit. The duality with the superhedging problem is trivial in the discrete case.

Our paper is organized as follows: After a mathematical rigorous introduction of the problem and the presentation of some well-known facts about the convex order, we present the linear programming formulation in Section \ref{sec:model}. In Section \ref{sec:spence} we define the generalized Spence Mirrlees condition and prove the optimality of the left-monotone transport plan. In the next section we present an algorithm which solves the primal problem. In Section \ref{sec:dual} we consider an algorithm which solves the dual problem. The last section contains an example.

\section{Preliminary Results and Original Problem}
Consider a financial market with one risk-free asset and one risky asset.
We allow only for two trading times which we denote by $0<t<T$.  The risk-free asset has no interest and its price is given by by $B=(B_t,B_T)=(1,1)$ with $B_0=1$ and for the risky asset we write $S=(S_t,S_T)=(X,Y)$ with $S_0 =s_0 \in \R$. Throughout we will assume that $X$ and $Y$ are discrete random variables and we consider a canonical construction, i.e. for the probability space we may choose $\Omega:=\Omega_1\times \Omega_2:= \{x_1,\ldots,x_N\}\times \{y_1,\ldots ,y_M\}$ and $X(\omega)=X(x,y)=x$ and $Y(\omega)=Y(x,y)=y.$ The $\sigma$-algebra is given by the power set. In what follows we denote by $\mathcal{P}(\Omega)$ the set of all probability measures on $\Omega$. Finally, we denote by $C_t(k)$ the price of a call-option with strike $k$ and maturity $t$, i.e. the payoff is given by $(X-k)^{+}$ and by $C_T(k)$ the price of a call-option with strike $k$ and maturity $T$, i.e. the payoff is given by $(Y-k)^{+}$. Calls are frequently traded and we assume that these prices can be observed at the market. The question in model-independent finance now is to find bounds on consistent prices for other derivatives on this market. A general payoff function of such a derivative in our framework is given by
  \[
       c:\Omega\to \R,\quad (x,y) \mapsto c(x,y),
  \]
  i.e.\ the payoff depends only on $S_t$ and $S_T$.
The no-arbitrage assumption generally implies the existence of a probability measure $\Q$ such that the price process $(X,Y)$ is a martingale under $\Q$ and the risk-neutral pricing formula holds, i.e. the price of a derivative with payoff $c$ is
\begin{equation}\label{eq:priceformula}
\Eop_\Q[c(X,Y)].
\end{equation}  An interesting result in this direction is given in \cite{bl} where it is shown that the observation of call prices $C_t(k)$  for all strikes $k$ implies the knowledge of the marginal at time $t$ of the consistent pricing measure $\Q$. More precisely the following is true.

  \begin{lemma}[{\cite[Sec. 2]{bl}}]\label{Lem Breeden Litzenbeger}
    Let $\Q \in \mathcal{P}(\Omega)$ be consistent with the price functions of call options, i.e. for $t,T$ and all $k \in \R$, we have
    \begin{eqnarray*}\label{Lem Breeden Litzenberger Consitency}
      C_t(k)       = \int_{\R}(x-k)_+ \Q(d (x,y)),\\
       C_T(k)       = \int_{\R}(y-k)_+ \Q(d (x,y)).\
    \end{eqnarray*}
    Then we have
    \begin{eqnarray*}\label{Lem Breeden Litzenberger Distribution}
	  \Q(X\leq k)=1+C_t'(k+),\\
	  \Q(Y\leq k)=1+C_T'(k+)
	\end{eqnarray*}
	for the distribution function of $X$ and $Y$ under $\Q$, where $C_t'(\cdot +)$ denotes the right side derivative of $C_t$ and likewise for $C_T$.
  \end{lemma}

Thus for pricing further derivatives we assume that the marginal distributions of the risky asset price process are known under the consistent pricing measure. But of course this does not imply the joint distribution $\Q$ which is necessary to compute general prices in \eqref{eq:priceformula}. However we have the further information that $(X,Y)$ has to be a martinagle under $\Q$ which restricts the distributions.

  \begin{definition}\label{Def MTP}
    Let $\mu,\nu$ be probability measures on $\Omega_1$  and  $\Omega_2$ respectively. The elements of the set
    \begin{eqnarray*}
      \mathcal{M}_2(\mu,\nu):=
       \left\{ \Q \in \mathcal{P}(\Omega)\mid \mu(A) = \Q(A\times \Omega_2), \nu(B) = \Q(\Omega_1\times B), \Eop_{\Q}[Y\mid X] = X  \ \Q-\text{a.s.}  \right\}
    \end{eqnarray*}
     are called \textit{martingale transport plans} or \textit{potential pricing measures}.
  \end{definition}

Hence $\mathcal{M}_2(\mu,\nu)$  is the set of all possible consistent pricing measures when we assume that $\mu$ is the distribution of $X$ and $\nu$ is the distribution of $Y$.
 Then we define for a given derivative with payoff $c$ the upper price bound problem
  \begin{equation}\label{Eq Upper Price Bound Standard Case}
    P(\mu,\nu):= \sup_{\Q \in \mathcal{M}_2(\mu,\nu)}\Eop_\Q[c(X,Y)].
    \end{equation}
    and the lower price bound problem
    \begin{equation}\label{Eq Lower Price Bound Standard Case}
   \underline P(\mu,\nu):= \inf_{\Q \in \mathcal{M}_2(\mu,\nu)}\Eop_\Q[c(X,Y)].
 \end{equation}
 In what follows we will concentrate on problem \eqref{Eq Upper Price Bound Standard Case}. Problem \eqref{Eq Lower Price Bound Standard Case} is in some sense symmetric (see Remark \ref{rem:lower}).
We will later see that in our setting this problem reduces to a simple linear program which can be solved efficiently. However in case $c$ satisfies a certain property, the solution of this linear program is given by a special structure and we will provide an algorithm for the solution.


In what follows when we consider measures $\mu$  on $\R$, we always assume that they are finite, i.e.   $\mu(\R)<\infty$ and that the integral exists, i.e.  $\int |x|\mu(dx)<\infty$. Also the next definition of the convex order is introduced for measures (see e.g. \cite{s51}) and not only probability measures as it is often done.

 \begin{definition}\label{Def Convex Order}
    Two measures $\mu,\nu $ on $\R$ are said to be in \textit{convex order}, denoted by $\mu\leq_c\nu$, if for any convex function $f:\R \to \R$ such that the integrals exist,
    \[
      \int_{\R} f(x) \mu( d x) \leq \int_{\R} f(x) \nu( d x).
    \]
      \end{definition}
Since both $f(x)=x$ and $f(x)=-x$ are convex as well as $f(x)=1$ and $f(x)=-1$, the property $\mu\leq_c\nu$ implies that $\int x\mu(dx)=\int x\nu(dx)$ and $\mu(\R)=\nu(\R)$.

  \begin{definition}\label{Def Call Option Price Func}
    For a measure  $\mu $ on $\R$ the corresponding \textit{call option price function}  is defined by
    \[
          C_\mu:\R \to \R_+, \quad  k \mapsto \int_{\R}(x-k)_{+} \mu( d x).
        \]
  \end{definition}

In arbitrage-free markets it is well-known that call option prices increase with the maturity, i.e. we have $C_\mu(k) \le C_\nu(k)$. Recall here that  $\mu$ is the distribution of $X$ and $\nu$ is the distribution of $Y$ for all $k$, otherwise arbitrage opportunities exist. Now we obtain the following relation which can be found in \cite{chong74}:

  \begin{lemma}\label{Prop Char Conv Ord MultiMarginal}
    Let $\mu, \nu$ be measures on $\R$ with $\frac{1}{\mu(\R)}\int x\mu(dx)= \frac{1}{\nu(\R)}\int x\nu(dx)$. Then the following are equivalent.
    \begin{enumerate}
          \item[1.] $\mu \leq_c  \nu$.
          \item[2.] $C_{\mu} \leq  C_{\nu}$.
    \end{enumerate}
      \end{lemma}

Thus the condition $\mu \leq_c  \nu$ is a natural assumption in our setting, since $\mu$ and $\nu$ are probability measures and hence $\mu(\R)=\nu(\R)=1$, $(X,Y)$ is a martingale and thus has constant expectation $\int x\mu(dx)= \int x\nu(dx)$ and due to no arbitrage requirements we have that $C_{\mu} \leq  C_{\nu}$.  Hence Lemma \ref{Prop Char Conv Ord MultiMarginal} implies the convex order between the marginal distributions $\mu$ and $\nu$.

Problem \eqref{Eq Upper Price Bound Standard Case} has a dual problem which is given by (see e.g. \cite{beiglbock2013model})
  \begin{align}\label{Eq Super Hedging Standard Case}
 & \inf_{(\varphi,\psi,h) \in \mathcal{D}_{2}}\left\{ \int_{\R}\varphi(x)\mu(d x) + \int_{\R}\psi(y)\nu(d y) \right\}
    \\ \notag
    &=\inf_{(\varphi,\psi,h) \in \mathcal{D}_{2}}\left\{ \Eop_{\mu}[\varphi(X)]+\Eop_{\nu}[\psi(Y)] \right\},
  \end{align}
  where
\begin{align*}
    \mathcal{D}_{2}:= \{(\varphi,\psi,h) \;|\; &\varphi,h :\Omega_1 \to \R, \psi : \Omega_2\to \R,
    \varphi(x) + \psi(y)+ h(x)(y-x) \geq c(x,y), \ (x,y) \in \Omega \}.
  \end{align*}
 This problem may be interpreted as finding the cheapest super-replication strategy for the payoff $c$.    Hedging strategies of this form are called semi-static hedging strategies, as $\varphi$ and $\psi$ may be interpreted as static investments in European options with maturity $t$ and $T$ respectively, while $h$ may be understood as a dynamic investment in the underlying asset.   Again we will see that this problem reduces to a linear program and it is indeed the dual linear program to \eqref{Eq Upper Price Bound Standard Case}.

\section{The Problem with Discrete Marginals}\label{sec:model}
In what follows we assume that $\mu$ and $\nu $ are discrete probability distributions, i.e. there are $N,M \in \N $ such that
    \[
        \mu=\sum_{j=1}^{N}\omega_{j}\delta_{x_j} \quad \text{ and }\quad  \nu=\sum_{i=1}^{M}\vartheta_i\delta_{y_i},
    \]
    where $\omega_j, \vartheta_i \geq 0 $, $x_j,y_i \in \R$ for all $j=1,\ldots,N $ and all $i=1,\ldots , M$ and $\sum_{j=1}^{N}\omega_j=\sum_{i=1}^{M}\vartheta_i=1$. $\delta_x$ is the Dirac measure in point $x$. Moreover we assume that $\mu \leq_c \nu$.

 Under this assumptions, martingale transport plans $\Q \in \mathcal{M}_{2}(\mu,\nu)$ are of the form
  \[
      \Q=\sum_{j=1}^{N}\sum_{i=1}^{M} q_{j,i}\delta_{(x_j,y_i)},
  \]
  where the following additional constraints have to be satisfied.
  \begin{enumerate}
       \item The masses of $\Q$ are non-negative, i.e. we have $q_{j,i} \geq 0$, for all $j=1,\ldots, N $ and all $i=1,\ldots,M$.
      \item The marginal distributions of $\Q$ are $\mu$ and $\nu$, i.e. we have
           \begin{align*}
                       \sum_{i=1}^{M}q_{j,i}=\omega_j, \quad   j=1,\ldots,N,
                       \\
                       \sum_{j=1}^{N} q_{j,i}=\vartheta_i, \quad    i=1,\ldots, M.
                     \end{align*}
          This implies $\sum_{j=1}^{N}\sum_{i=1}^{M}q_{j,i}=1$ such that $\Q$ is indeed a probability measure.
      \item The measure $\Q$ satisfies the martingale condition. Transferring the classic condition $\Eop_{\Q}\left[Y\mid X \right]=X$ to the discrete situation, we have \[
            \sum_{i=1}^{M}\frac{q_{j,i}}{\omega_j}y_i=x_j, \quad  j = 1 ,\ldots, N,
          \]
          as $\frac{q_{j,i}}{\omega_j} = \frac{\Q((X,Y)=(x_j,y_i))}{\Q(X=x_j)}$ is the  conditional distribution $\Q(Y=y_i|X=x_j)$. Rewriting this as
           \[
                       \sum_{i=1}^{M}q_{j,i}y_i=\omega_jx_j\iff \sum_{i=1}^{M}q_{j,i}y_i - \sum_{i=1}^{M}q_{j,i}x_j = 0 \iff \sum_{i=1}^{M} q_{j,i}(y_i-x_j)=0,
                     \]
                     we find an alternative condition.
       \end{enumerate}

   Altogether, the upper price bound problem in \eqref{Eq Upper Price Bound Standard Case} reduces in the discrete case to the following  linear program
  \begin{align}\label{Prob Upper Price Bound Discrete}\tag{P}
    \max \quad   &\sum_{j=1}^{N} \sum_{i=1}^{M}q_{j,i}c(x_j,y_i)  := \sum_{j=1}^{N} \sum_{i=1}^{M}q_{j,i}c_{j,i}
    \\ \notag
    \text{s.t.} \quad  &\sum_{i=1}^{M}q_{j,i}=\omega_j, \quad  j=1,\ldots, N ,
    \\ \notag
   &\sum_{j=1}^{N}q_{j,i}=\vartheta_i, \quad  i=1,\ldots,M,
    \\ \notag
    & \sum_{i=1}^{M} q_{j,i}(y_i-x_j)=0, \quad j =1 ,\ldots, N,
    \\ \notag
    &   q_{j,i} \geq 0, \quad j=1,\ldots,N, i=1,\ldots, M.
 \end{align}
 The first two equations guarantee that we have the correct  marginal distributions, the last equation is the martingale condition.
The dual problem is given by
  \begin{align}\label{Prob Super Hedging Discrete}\tag{D}
    \min \quad &\sum_{j=1}^{N}\omega_j \varphi_j + \sum_{i=1}^{M} \vartheta_i \psi_i
    \\
    \notag \text{s.t.}\quad  &\varphi_j + \psi_i + h_j (y_i-x_j)\geq c_{j,i},\quad  j=1,\ldots,N, i=1,\ldots,M,
    \\
    \notag& \varphi_j,h_j,\psi_i \in \R, \quad  j=1,\ldots,N, i=1,\ldots,M,
  \end{align}
and can be interpreted as  a super hedging problem.

The next result is due to \cite{s51} and has later been generalized  in various ways (see also \cite{str}).

\begin{lemma}\label{lem:sherman}
It holds that $\mu \leq_c \nu$ if and only if there exists $q_{j,i} \geq 0$ such that
  \begin{align}
    &\sum_{i=1}^{M}q_{j,i}=\omega_j, \quad  j=1,\ldots, N ,
    \\ \notag
   &\sum_{j=1}^{N}q_{j,i}=\vartheta_i, \quad  i=1,\ldots,M,
    \\ \notag
    & \sum_{i=1}^{M} q_{j,i}(y_i-x_j)=0, \quad j =1 ,\ldots, N.
 \end{align}

\end{lemma}

Thus, we immediately obtain the following existence result.

\begin{theorem}\label{theo:existence}
The linear programs (P) and (D) have optimal solutions $q^*$ and $(\varphi^*,h^*,\psi^*)$ and $\sum_{i=1}^{M}q^*_{j,i}c_{j,i} = \sum_{j=1}^{N}\omega_j \varphi^*_j + \sum_{i=1}^{M} \vartheta_i \psi^*_i$, i.e. the optimal value of the objective functions coincide.
\end{theorem}

\begin{proof}
Lemma \ref{lem:sherman} implies that (P) has feasible points. Since $q_{j,i} \in [0,1]$ the target function is bounded and thus by the classical existence theorem for linear programs an optimal solution for (P) exists. But by the strong duality theorem of linear programming this implies that (D) has a solution and the optimal values of the objective functions coincide.
\end{proof}

\section{Left-Monotonicity and a kind of Martingale Spence Mirrlees Condition}\label{sec:spence}
We will consider the linear program (P) with functions $c$ having a special property.
In this case it is possible to prove that a solution $\Q^*=q^*$ of (P) has a certain property.
The following definition is due to \cite{beiglboeck2016problem}.

 \begin{definition}\label{def:leftmo}
  A martingale transport plan ${\Q \in \mathcal{M}_2(\mu,\nu)}$ is called \textit{left monotone}, if there is a set $\Gamma \subseteq \Supp(\mu)\times \Supp(\nu)$ with ${\Q( \Gamma)=1}$ and such that for $(x,y^-),(x,y^+),(x',y')\in \Gamma$ with $x<x'$, we have $y' \notin (y^-,y^+)$.
  \end{definition}

    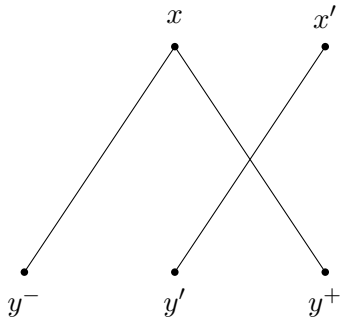
\begin{figure}[!htbp]\begin{center}

       \begin{tikzpicture}[scale=1]

          \node  at (0,0.4) {$x$};
          \node at (2,0.4) {$x'$};
         \fill (0,0) circle (0.05);
         \fill (2,0) circle (0.05);

          \node at (-2,-3.4) {$y^{-}$};
      \node at (0,-3.4) {$y'$};
          \node at (2,-3.4) {$y^{+}$};
           \fill (-2,-3) circle (0.05);
                  \fill (0,-3) circle (0.05);
                  \fill (2,-3) circle (0.05);

                  \draw(0,0)--(-2,-3);
                     \draw(0,0)--(2,-3);
                     \draw(2,0)--(0,-3);

        \end{tikzpicture}
        \caption{Forbidden configuration for left monotonicity.}
  \end{center}
    \end{figure}

\begin{theorem}\label{theo:lmopt}
Let $c:\R^2 \to \R$ be a payoff function such that for all $x'>x$ and $y^{+}>y'>y^{-}$ with $x,x' \in \Supp(\mu)$ and $ y^{-}, y',y^{+} \in \Supp(\nu)$ we have
   \begin{equation}\label{eq:optcondition}
      \lambda\left[ c(x',y^{+})-c(x,y^{+}) \right] + (1-\lambda)\left[ c(x',y^{-})-c(x,y^{-}) \right] - \left[ c(x',y')-c(x,y') \right] > 0,
    \end{equation}
    where $\lambda = \frac{y'-y^{-}}{y^{+}-y^{-}}\in[0,1]$. Then an optimal $q^*=\Q^*$ of (P) is left-monotone. The (partial) converse is also true: If  there exists a unique solution of (P) which is left-monotone, then $c$ has to satisfy \eqref{eq:optcondition}. 

\end{theorem}

\begin{proof}
Suppose $x,x' \in \Supp(\mu)$ and $ y^{-}, y',y^{+} \in \Supp(\nu)$ are such that $x'>x$ and $y^{+}>y'>y^{-}$ and suppose that ${\Q \in \mathcal{M}_2(\mu,\nu)}$ satisfies
\begin{eqnarray*}
\Q(x,y^-) := \theta^->0\\
\Q(x,y^+) := \theta^+>0\\
\Q(x',y') := \theta'>0,
\end{eqnarray*}
i.e. the condition in Definition \ref{def:leftmo} is not satisfied.
Let $\lambda = \frac{y'-y^{-}}{y^{+}-y^{-}} $ which implies that $y'=\lambda y^++(1-\lambda)y^-$.  We will now define a new $\tilde{\Q}\in \mathcal{M}_2(\mu,\nu)$ such that the objective function of (P) attains a higher value which implies that $\Q$ cannot be optimal.
We consider the following new arrangement $\tilde{\Q}$ of the probability mass
\begin{eqnarray*}
\tilde{\Q}(x',y^-) &:=& \Q(x',y^-)+\theta'(1-\lambda)\\
\tilde{\Q}(x',y^+) &:=& \Q(x',y^+)+\theta'\lambda\\
\tilde{\Q}(x',y') &:=& 0\\
\tilde{\Q}(x,y^-) &:=& \theta^--\theta'(1-\lambda)\\
\tilde{\Q}(x,y^+) &:=& \theta^+-\theta'\lambda\\
\tilde{\Q}(x,y') &:=& \Q(x,y')+\theta'
\end{eqnarray*}
All other assignments are left unchanged.
In a first step we assume here that $\theta^--\theta'(1-\lambda)\ge0$ and $\theta^+-\theta'\lambda\ge0$ and claim that $\tilde{\Q}$ is again admissible for (P). That the marginal distributions are preserved is easy to see.  In order to show that the martingale condition holds we have to show that the values $\sum_i \Q(x,y_i)y_i$ and $\sum_i \Q(x',y_i)y_i$ do not change under the new probability assignment. First consider $x$. Here we obtain
\begin{eqnarray*}
\sum_i \Q(x,y_i)y_i-\sum_i \tilde{\Q}(x,y_i)y_i &=&
y^- \theta^- + y^+ \theta^+ -\big( y^- (\theta^--\theta'(1-\lambda)) + y^+(\theta^+-\theta'\lambda)+y'\theta'\big)\\
&=&  y^- \theta^- + y^+ \theta^+ -\big(y^- \theta^-+ y^+\theta^++\theta'(y'-y^-(1-\lambda)-y^+\lambda) \big)\\
&=&0
\end{eqnarray*}
which is true since the term in the inner brackets vanishes due to the definition of $\lambda$. Next consider $x'$. Here we obtain:
\begin{eqnarray*}
\sum_i \Q(x',y_i)y_i-\sum_i \tilde{\Q}(x',y_i)y_i &=&
y' \theta'-\big( y^- \theta' (1-\lambda)  + y^+\theta'\lambda\big) \\
&=& \theta' \big[ y'-(y^-(1-\lambda)+y^+\lambda)\big]=0
\end{eqnarray*}
which is again true by the definition of $\lambda$. Thus we have that $\tilde{\Q}\in \mathcal{M}_2(\mu,\nu)$. Now we finally prove that the value of the objective function under $\tilde{\Q}$ is larger than under $\Q$. We obtain for the difference of the objective functions:
\begin{eqnarray*}
\sum_{j=1}^{N} \sum_{i=1}^{M}\big(\tilde{\Q}(x_j,y_i)-\Q(x_j,y_i)\big)c(x_j,y_i) &=& c(x',y^-) \theta'(1-\lambda)+ c(x',y^+) \theta'\lambda-c(x',y')\theta'\\[-0.4cm]
&& -c(x,y^-) \theta'(1-\lambda)-c(x,y^+) \theta'\lambda+c(x,y')\theta'\\
&=& \lambda\theta' \left[ c(x',y^{+})-c(x,y^{+}) \right] \\&&+ (1-\lambda)\theta'\left[ c(x',y^{-})-c(x,y^{-}) \right] \\
&&- \theta'\left[ c(x',y')-c(x,y') \right]> 0
\end{eqnarray*}
by our assumption. Also note that the new assignment dissolves the forbidden configuration. In case the conditions $\theta^--\theta'(1-\lambda)\ge0$ and $\theta^+-\theta'\lambda\ge0$ are not satisfied, we keep part of the probability mass $\theta'$ on $(x',y')$ and shift only the amount $\tilde{\theta}$ such that $\min\{\theta^--\tilde{\theta}(1-\lambda), \theta^+-\tilde{\theta}\lambda \}=0$ in the same way as before. In this case we dissolve  again the forbidden configuration and obtain a higher value for the objective function using the same arguments as before. From the proof so far, we see that the converse statement is also true: Suppose there are $x,x' \in \Supp(\mu)$ and $ y^{-}, y',y^{+} \in \Supp(\nu)$  such that $x'>x$ and $y^{+}>y'>y^{-}$ and $c$ does not satisfy \eqref{eq:optcondition} on these points. Then inverting the transport that we have constructed in the first part we see that $\Q$ is not worse than $\tilde{\Q}$ which implies that $\tilde{\Q}$ cannot be the unique solution.
\end{proof}


The following notion has been introduced in \cite{henry2016explicit} and it has been shown there that this property of $c$ implies that the optimal transport plan is left-monotone in the case of continuous margins. The same has been shown in \cite{beiglbock2017monotone} for general margins. In the next lemma we show that this condition implies \eqref{eq:optcondition}. Of course \eqref{eq:optcondition} does not need a differentiable function $c$ and is thus more general.

  \begin{definition}\label{Def Martingale Spence Mirrlees}
    A function $c:\R^2 \to \R$ satisfies the \textit{martingale Spence Mirrlees condition}, if the partial derivative $c_{xyy}$ exists and satisfies $c_{xyy}>0$.
  \end{definition}

 \begin{lemma}
    Let $c:\R^2 \to \R$ be a function satisfying the martingale Spence Mirrlees condition. Then \eqref{eq:optcondition} holds for all $x'>x$, $y^{+}>y'>y^{-}$.
  \end{lemma}

  \begin{proof}
 First note that
    \[
      \lambda\left[ c(x',y^{+})-c(x,y^{+}) \right] + (1-\lambda)\left[ c(x',y^{-})-c(x,y^{-}) \right] - \left[ c(x',y')-c(x,y') \right] > 0,
    \]
         is  equivalent to
    \begin{align*}
      &\lambda\left[ c(x',y^{+})-c(x',y')-c(x,y^{+})+c(x,y') \right]
      \\
      &-  (1-\lambda)\left[ c(x',y') - c(x',y^{-})-c(x,y')+c(x,y^{-}) \right]>0.
    \end{align*}

    If we plug in $\lambda = \frac{y'-y^{-}}{y^{+}-y^{-}} $ and multiply by $y^{+}-y^{-}$, we obtain
    \begin{align}\nonumber
    &    {[c(x',y^{+})-c(x',y')-c(x,y^{+})+c(x,y')]}(y'-y^{-}) \\ \label{eq:msmc}
       -&  {[c(x',y') - c(x',y^{-})-c(x,y')+c(x,y^{-})]}(y^{+}-y')>0.
     \end{align}
Since $c_{xyy}>0$ and $s \ge y' \ge u$ for all $s \in [y',y^+]$ and $u \in [y^-,y']$ we obtain:
\begin{align*}
  0 &< \int_{x}^{x'}\int_{y'}^{y^{+}}\int_{y^{-}}^{y'}\int_{u}^{s} c_{xyy}(t,v) \mathrm{d}v\mathrm{d}u\mathrm{d}s\mathrm{d}t 
  \\&= \int_{x}^{x'}\int_{y'}^{y^{+}}\int_{y^{-}}^{y'} c_{xy}(t,s)-c_{xy}(t,u) \mathrm{d}u\mathrm{d}s\mathrm{d}t
  \\&= \int_{x}^{x'}\int_{y'}^{y^{+}} c_{xy}(t,s)(y'-y^{-})-(c_{x}(t,y')-c_{x}(t,y^{-})) \mathrm{d}s\mathrm{d}t
  \\&= \int_{x}^{x'} (c_{x}(t,y^{+})-c_{x}(t,y'))(y'-y^{-})-(c_{x}(t,y')-c_{x}(t,y^{-}))(y^{+}-y') \mathrm{d}t
  \\&=  [c(x',y^{+})-c(x,y^{+})-(c(x',y')-c(x,y'))](y'-y^{-})
    \\&\quad    - [c(x',y')-c(x,y') - (c(x',y^{-})-c(x,y^{-}))](y^{+}-y').
\end{align*}
Comparing it with \eqref{eq:msmc} yields the assertion.
  \end{proof}

\begin{corollary}
If $c$ satifies the martingale Spence Mirrlees condition, then the optimal $q^*=\Q^*$ of (P) is left-monotone.
\end{corollary}

\begin{remark}\label{rem:lower}
As far as the lower price bound in problem \eqref{Eq Lower Price Bound Standard Case} is concerned, we obtain a similar result: If the martingale Spence Mirrlees condition is satisfied, then a right-monotone transport plan is optimal.  A $\Q \in \mathcal{M}_2(\mu,\nu)$ is called \textit{right monotone}, if there is a Borel set $\Gamma  \subset
\Supp(\mu)\times \Supp(\nu) $ with $\Q( \Gamma)=1$ and such that for all $(x,y_1),(x,y_2),(x',y')\in \Gamma$ with $x>x'$, we have $y' \notin (y_1,y_2)$. In case $c_{xyy}<0$, the optimality properties reverse, i.e. the right monotone transport plan is optimal for \eqref{Eq Upper Price Bound Standard Case} and the left monotone transport plan for \eqref{Eq Lower Price Bound Standard Case}.
\end{remark}

\section{An Algorithm for the Construction of the Optimal Solution of (P)}\label{sec:algo}
In this section we derive an algorithm which computes an optimal solution of (P). We assume that $\mu,\nu$ satisfy the conditions $\mu \leq_c \nu$, $\mu= \sum_{j=1}^{N}\omega_j \delta_{x_j}$ and $\nu=\sum_{i=1}^{M}\vartheta_i\delta_{y_i}$ and the function $c$ has property \eqref{eq:optcondition}. Theorem \ref{theo:lmopt} and Theorem \ref{theo:existence} imply that a left-monotone transport plan exists. We denote by $\Q_l(\mu,\nu)$ any such left-monotone transport plan.

Before we formalize the algorithm to determine $\Q_l(\mu,\nu)$, let us develop an intuition how it could work. As $\Q_l(\mu,\nu)$ is in particular a coupling of its margins $\mu$ and $\nu$, we have to specify how much of the mass of each $x_j$ is transported to each $y_i$. Of course, this transport specification has to satisfy some prerequisites. Clearly, the coupling can be done for one atom of $\mu$ at a time. For this purpose, by (a)-(c) in Section \ref{sec:model}, it is obvious how an algorithm should proceed in order to satisfy the marginal and martingale conditions. Thus, the central idea stems from the desired left monotonicity. In order not to introduce a contradiction to left monotonicity when coupling an atom of $\mu$ with atoms of $\nu$, we should couple the smallest atom of $\mu$ martingale-consistently with atoms of $\nu$ such that no mass of $\nu$ is left over in the convex hull of those atoms. This does indeed guarantee that in a next step, coupling the second smallest atom of $\mu$, no contradictions to the left monotonicity can be introduced.

In order to get a non-trivial problem we assume now that $\mu$ and $\nu$ satisfy the following conditions.
    \begin{enumerate}
       \item $x_1<\ldots <x_N$, $y_1< \ldots < y_M$.
      \item $\mu \neq \nu$.
      \item $\mu$ has at least two different atoms of positive mass.
    \end{enumerate}

First we would like to note the following obvious fact.

\begin{lemma}\label{lem:cximpl}
The assumption $\mu\le_c\nu$ implies that for all $x_j\in \Supp(\mu)$ there either exists $\ell \in \{1,\ldots,M\}$ such that $x_j=y_\ell$ or there exists an $\ell \in \{1,\ldots,M-1\}$ such that $y_\ell<x_j<y_{\ell+1}$.
\end{lemma}

\begin{proof}
This property follows directly from the martingale condition:
$$\sum_{i=1}^M q_{j,i}(y_i-x_j)=0$$
because the atoms of $\nu$ cannot all lay on one side of $x_j$.
\end{proof}

The next two properties have to be satisfied by any left-monotone transport plan.

  \begin{lemma}\label{lem:c1}
    Suppose $x_1 \in \Supp(\nu)$, say $x_1=y_\ell$ for some $\ell \in \{1,\ldots,M\}.$
    \begin{enumerate}
      \item Then $(x_1,y_\ell)$ is an atom of $\Q_{l}(\mu,\nu)$.
      \item Further suppose $\ell =1$. Then $\omega_1 \leq \vartheta_1$.
    \end{enumerate}
  \end{lemma}

  \begin{proof}
    \begin{enumerate}
      \item In order to get a contradiction, assume that $x_1$ is not coupled with $y_\ell$ under $\Q_{l}(\mu,\nu)$. Then there are $y^-,y^+ \in \Supp(\nu)$ with $y^-<y_\ell < y^+$ and such that $x_1$ is coupled with $y^-$ and $y^+$. Also, there is some $x' \in \Supp(\mu)$ with $x'>x_1$ and such that $x'$ is coupled with $y_{\ell}$. This contradicts the left monotonicity.
      \item In order to get a contradiction, assume that $\omega_1 >\vartheta_1$. After coupling $x_1$ with $y_1$ there is mass of at least $\omega_1 - \vartheta_1$ left in $x_1$. This has then to be coupled with some atoms from $\Supp(\nu)\setminus \{y_1\}$. As $y_i > x_1$ for all $i=2,\ldots,M$, this contradicts the martingale property.
    \end{enumerate}
      \end{proof}

 \begin{lemma}\label{lem:c2}
 Suppose $y_{\ell}<x_1 < y_{\ell+1}$ for some $\ell \in \{1,\ldots,M-1\}$. Then $(x_1,y_\ell)$ and $(x_1,y_{\ell +1})$ are atoms of $\Q_{l}(\mu,\nu)$.
 \end{lemma}

 \begin{proof}
   In order to get a contradiction, assume without loss of generality that $x_1$ is not coupled with $y_{\ell+1}$. Then there is a $y^+ \in \Supp(\nu)$ with $y^+ > y_{\ell +1}$ and such that $x_1$ is coupled with at least $y_{\ell}$ and $y^+$. Also, there is an $x' \in \Supp(\mu)$ with $x' >x_1$ and such that $x'$ is coupled with $y_{\ell+1}$. This contradicts the left monotonicity.
 \end{proof}

Before we start with the algorithm to construct a left-monotone transport plan we show that it is unique. This has already been shown in \cite{beiglboeck2016problem} for general margins using  rather complex methods. In contrast, our proof is very easy and does not rely on sophisticated concepts.

\begin{theorem}
There exists one unique left-monotone transport plan.
\end{theorem}

\begin{proof}
Suppose  there are at least two different left-monotone transport plans $q, \tilde{q} \in \mathcal{M}_2(\mu,\nu)$. Let $x_{j^*}$ be smallest atom of $\mu$ where there is an $i^{*}\in \{1,\ldots, M\}$ such that $q_{j^*,i^{*}}\neq \tilde{q}_{j^*,i^{*}}.$ We eliminate the variables $q_{j,i}, \tilde{q}_{j,i}$ for $j<j^*$ and $i=1,\ldots,M$ from the transport problem and consider the remaining constraints which for $q$ read:
  \begin{align*}
  &\sum_{i=1}^{M}q_{j,i}=\omega_j, \quad  j=j^*,\ldots, N ,
    \\ \notag
   &\sum_{j=j^*}^{N}q_{j,i}=\vartheta_i - \sum_{j=1}^{j^*-1}q_{j,i} , \quad  i=1,\ldots,M,
    \\ \notag
    & \sum_{i=1}^{M} q_{j,i}(y_i-x_j)=0, \quad j =j^* ,\ldots, N,
    \\ \notag
    &   q_{j,i} \geq 0, \quad j=j^*,\ldots,N, i=1,\ldots, M.
 \end{align*}
The remaining transport plans $q$ and $\tilde{q}$ are again left-monotone. Now consider the atom $x_{j^*}$  with the probability mass $\omega_{j^*}$. By the left-monotonicity $x_{j^*}$  has to be coupled with atoms $y_\ell <y_{\ell+1}<\ldots <y_{\ell+m}$ of $\nu$ such that  $q_{j^*,k}= \vartheta_k$ for $k=\ell+1,\ldots ,\ell+m-1$, $ \sum_{k=0}^m q_{j^*,\ell+k}=\omega_{j^*}$ and $$ x_{j^*} \omega_{j^*} = \sum_{k=0}^m q_{j^*,\ell+k} y_{\ell+k}.$$
However these conditions imply a unique solution. Thus $q$ and $\tilde{q}$ cannot be different.
\end{proof}

Next we write down an algorithm which produces the left-monotone transport plan. Note that Lemma \ref{lem:cximpl} implies that once we fix an $x_j$ then only two situations can occur: either  there exists an $\ell \in \{1,\ldots,M\}$ such that $x_j=y_\ell$ or there exists an $\ell \in \{1,\ldots,M-1\}$ such that $y_\ell<x_j<y_{\ell+1}$. Note that the algorithm works in a recursive way: the mass from $\mu$ to $\nu$ is transported step by step. After the transport we continue with the remaining masses. In particular the index $1$ always refers to the smallest remaining atom of  $\mu$ and thus may be different from the real index.\\

{\bf Algorithm (primal problem):} Start with $x_1$.
\begin{description}
\item[Case I]  There exists $\ell \in \{1,\ldots,M\}$ such that $x_1=y_\ell$
\begin{description}
\item[Case 1] $\omega_1\le \vartheta_\ell$.

Then set $q_{1,\ell}:= \omega_1, q_{1,i}:= 0, i\neq \ell$ and define $ \vartheta_\ell := \vartheta_\ell-\omega_1, \omega_1 := 0$. Cross $x_1$ (and possibly $y_{\ell}$) from the list. All other probability masses are kept unchanged. Continue with the smallest remaining atom of $\mu$.

   \begin{figure}[!htbp]\begin{center}

       \begin{tikzpicture}[scale=1]

               \draw(0,0)--(7,0);
               \draw(0,2)--(7,2);
              \draw(1,2.1)--(1,1.9);
         \node  at (1,1.7) {$x_1$};
          \node at (1,2.3) {$0$};
            \draw(1,0.1)--(1,-0.1);
           \node  at (1,0.3) {$y_\ell$};
            \node at (1,-0.3) {$\vartheta_\ell-\omega_1$};
         { \color{red}{   \draw(1,0)--(1,2);}
            \node  at (1.3,1) {$\omega_1$};}
            \draw(2,2.1)--(2,1.9);
           \node  at (2,1.7) {$x_2$};
          \draw(5,2.1)--(5,1.9);
           \node  at (5,1.7) {$x_N$};
                  \draw(6,0.1)--(6,-0.1);
           \node  at (6,0.3) {$y_M$};

        \end{tikzpicture}
        \caption{Case I, Case 1. The figure shows the new probability masses of the involved atoms.}
  \end{center}
    \end{figure}
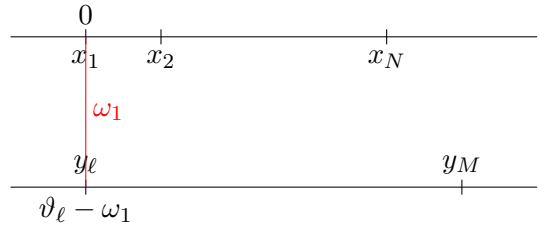

 \item[Case 2] $\omega_1> \vartheta_\ell$.

Then set $q_{1,\ell}:= \vartheta_\ell, q_{j,\ell}:= 0, j\neq 1$ and define $\omega_1 := \omega_1-\vartheta_\ell, \vartheta_\ell := 0$. Cross $y_\ell$ from the list. All other probability masses are kept unchanged. Continue with the smallest remaining atom of $\mu$.

      \begin{figure}[!htbp]\begin{center}

       \begin{tikzpicture}[scale=1]

               \draw(0,0)--(7,0);
               \draw(0,2)--(7,2);
              \draw(1,2.1)--(1,1.9);
         \node  at (1,1.7) {$x_1$};
          \node at (1,2.3) {$\omega_1-\vartheta_\ell$};
            \draw(1,0.1)--(1,-0.1);
           \node  at (1,0.3) {$y_\ell$};
            \node at (1,-0.3) {$0$};
         { \color{red}{   \draw(1,0)--(1,2);}
            \node  at (1.3,1) {$\vartheta_\ell$};}
            \draw(2,2.1)--(2,1.9);
           \node  at (2,1.7) {$x_2$};
          \draw(5,2.1)--(5,1.9);
           \node  at (5,1.7) {$x_N$};
                  \draw(6,0.1)--(6,-0.1);
           \node  at (6,0.3) {$y_M$};

        \end{tikzpicture}
        \caption{Case I, Case 2. The figure shows the new probability masses of the involved atoms}
  \end{center}
    \end{figure}
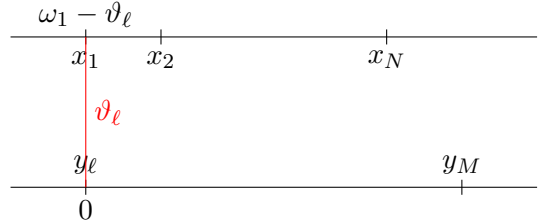

\end{description}

\item[Case II]  There exists $\ell \in \{1,\ldots,M-1\}$ such that $y_\ell<x_1<y_{\ell+1}$.

Then there exist $\vartheta'_\ell, \vartheta'_{\ell+1} \ge 0$ such that
\begin{eqnarray}
\label{eq:leq} \vartheta'_\ell + \vartheta'_{\ell+1} &=& \omega_1\\ \nonumber
 \vartheta'_\ell  y_\ell +\vartheta'_{\ell+1} y_{\ell+1} &=& \omega_1x_1.
\end{eqnarray}
Note that these linear equations have a unique solution.
\begin{description}
 \item[Case 1] $\vartheta'_\ell \le \vartheta_\ell, \vartheta'_{\ell+1}\le \vartheta_{\ell+1}.$

 Then set $q_{1,\ell} := \vartheta'_\ell, q_{1,\ell+1}:= \vartheta'_{\ell+1} , q_{1,i}:=0, i\neq \ell,\ell+1$ and define $\omega_1 := 0, \vartheta_\ell := \vartheta_\ell-\vartheta'_\ell, \vartheta_{\ell+1}:= \vartheta_{\ell+1}-\vartheta'_{\ell+1}.$ Cross $x_1$ from the list (and possibly $y_{\ell}, y_{\ell+1}$). All other probability masses are kept unchanged. Continue with the smallest remaining atom of $\mu$.

      \begin{figure}[!htbp]\begin{center}

       \begin{tikzpicture}[scale=1]

               \draw(-0.5,0)--(7,0);
               \draw(-0.5,2)--(7,2);
              \draw(1,2.1)--(1,1.9);
         \node  at (1,1.7) {$x_1$};
          \node at (1,2.3) {$0$};
            \draw(-0.1,0.1)--(-0.1,-0.1);
           \node  at (-0.1,0.3) {$y_\ell$};
            \node at (-0.1,-0.3) {$\vartheta_\ell-\vartheta'_\ell$};
                  \draw(2.1,0.1)--(2.1,-0.1);
           \node  at (2.1,0.3) {$y_{\ell+1}$};
            \node at (2.1,-0.3) {$\vartheta_{\ell+1}-\vartheta'_{\ell+1}$};

         { \color{red}{   \draw(-0.1,0)--(1,2);}
            \node  at (0,1.1) {$\vartheta'_\ell$};
            \draw(2.1,0)--(1,2);
             \node  at (2,1.1) {$\vartheta'_{\ell+1}$};}
           \node  at (2,1.7) {$x_2$};
          \draw(5,2.1)--(5,1.9);
           \node  at (5,1.7) {$x_N$};
                  \draw(6,0.1)--(6,-0.1);
           \node  at (6,0.3) {$y_M$};

        \end{tikzpicture}
        \caption{Case II, Case 1. The figure shows the new probability masses of the involved atoms}
  \end{center}
    \end{figure}
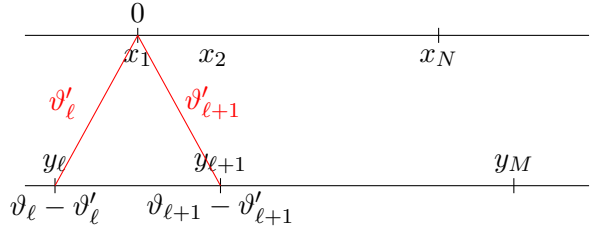
 \item[Case 2] $\vartheta'_\ell > \vartheta_\ell, \vartheta'_{\ell+1}\le \vartheta_{\ell+1}.$

 In this case we multiply the linear equations of  \eqref{eq:leq} with $\frac{\vartheta_\ell}{\vartheta'_\ell}$ to obtain:
 \begin{eqnarray*}
 \vartheta_\ell + \frac{\vartheta_\ell}{\vartheta'_\ell} \vartheta'_{\ell+1} &=& \omega_1 \frac{\vartheta_\ell}{\vartheta'_\ell}\\
 \vartheta_\ell  y_\ell +\frac{\vartheta_\ell}{\vartheta'_\ell} \vartheta'_{\ell+1} y_{\ell+1} &=& \omega_1 \frac{\vartheta_\ell}{\vartheta'_\ell}x_1.
\end{eqnarray*}

 Then set $q_{1,\ell} := \vartheta_\ell, q_{1,\ell+1}:= \frac{\vartheta_\ell}{\vartheta'_\ell}\vartheta'_{\ell+1} , q_{i,\ell}:=0, i\neq 1$ and define $\omega_1 := \omega_1(1-\frac{\vartheta_\ell}{\vartheta'_\ell}), \vartheta_{\ell+1}:= \vartheta_{\ell+1}-\frac{\vartheta_\ell}{\vartheta'_\ell}\vartheta'_{\ell+1}$, $\vartheta_\ell := 0$. Cross $y_\ell$ from the list. All other probability masses are kept unchanged. Continue with the smallest remaining atom of $\mu$. Note that after this step there is still probability mass on $x_1$ and $y_{\ell+1}$ which could result in the next step in another transport from $x_1$ to $y_{\ell+1}$. At the end these transported masses have to be added.

  \begin{figure}[!htbp]\begin{center}

       \begin{tikzpicture}[scale=1]

               \draw(-0.5,0)--(7,0);
               \draw(-0.5,2)--(7,2);
              \draw(1,2.1)--(1,1.9);
         \node  at (1,1.7) {$x_1$};
          \node at (1,2.3) {$\omega_1(1-\frac{\vartheta_\ell}{\vartheta'_\ell})$};
            \draw(-0.1,0.1)--(-0.1,-0.1);
           \node  at (-0.1,0.3) {$y_\ell$};
            \node at (-0.1,-0.3) {$0$};
                  \draw(2.1,0.1)--(2.1,-0.1);
           \node  at (2.1,0.3) {$y_{\ell+1}$};
            \node at (2.1,-0.5) {$\vartheta_{\ell+1}-\frac{\vartheta_\ell}{\vartheta'_\ell}\vartheta'_{\ell+1}$};

         { \color{red}{   \draw(-0.1,0)--(1,2);}
            \node  at (0,1.1) {$\vartheta_\ell$};
            \draw(2.1,0)--(1,2);
             \node  at (2.2,1.1) {$\frac{\vartheta_\ell}{\vartheta'_\ell}\vartheta'_{\ell+1}$};}
           \node  at (2,1.7) {$x_2$};
          \draw(5,2.1)--(5,1.9);
           \node  at (5,1.7) {$x_N$};
                  \draw(6,0.1)--(6,-0.1);
           \node  at (6,0.3) {$y_M$};

        \end{tikzpicture}
        \caption{Case II, Case 2. The figure shows the new probability masses of the involved atoms}
  \end{center}
    \end{figure}
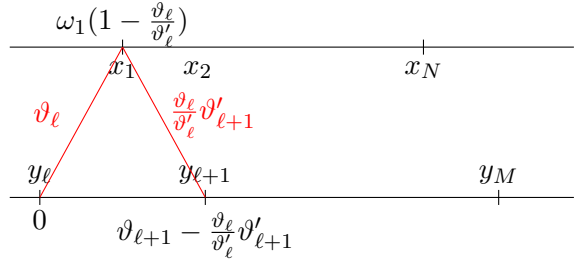

  \item[Case 3] $\vartheta'_\ell \le \vartheta_\ell, \vartheta'_{\ell+1}> \vartheta_{\ell+1}.$

 In this case we multiply the linear equation \eqref{eq:leq} with $\frac{\vartheta_{\ell+1}}{\vartheta'_{\ell+1}}$ to obtain:
 \begin{eqnarray*}
 \vartheta'_\ell \frac{\vartheta_{\ell+1}}{\vartheta'_{\ell+1}} + \vartheta_{\ell+1} &=& \omega_1 \frac{\vartheta_{\ell+1}}{\vartheta'_{\ell+1}}\\
 \vartheta'_\ell \frac{\vartheta_{\ell+1}}{\vartheta'_{\ell+1}}  y_\ell + \vartheta_{\ell+1} y_{\ell+1} &=& \omega_1 \frac{\vartheta_{\ell+1}}{\vartheta'_{\ell+1}}x_1.
\end{eqnarray*}

 Then set $q_{1,\ell} := \frac{\vartheta_{\ell+1}}{\vartheta'_{\ell+1}} \vartheta'_\ell, q_{1,\ell+1}:=\vartheta_{\ell+1} , q_{i,\ell+1}:=0, i\neq 1$ and define $\omega_1 := \omega_1(1-\frac{\vartheta_{\ell+1}}{\vartheta'_{\ell+1}}), \vartheta_\ell := \vartheta_\ell- \frac{\vartheta_{\ell+1}}{\vartheta'_{\ell+1}}\vartheta'_\ell, \vartheta_{\ell+1}:= 0.$ Cross $y_{\ell+1}$ from the list. All other probability masses are kept unchanged. Continue with the smallest remaining atom of $\mu$. Note that after this step there is still probability mass on $x_1$ and $y_{\ell}$ which could result in the next step in another transport from $x_1$ to $y_{\ell}$. At the end these transported masses have to be added.
  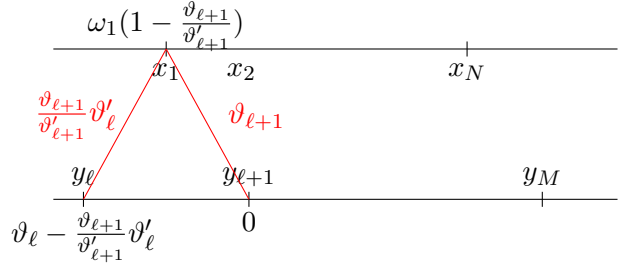
\begin{figure}[!htbp]\begin{center}

       \begin{tikzpicture}[scale=1]

               \draw(-0.5,0)--(7,0);
               \draw(-0.5,2)--(7,2);
              \draw(1,2.1)--(1,1.9);
         \node  at (1,1.7) {$x_1$};
          \node at (1,2.3) {$\omega_1(1-\frac{\vartheta_{\ell+1}}{\vartheta'_{\ell+1}})$};
            \draw(-0.1,0.1)--(-0.1,-0.1);
           \node  at (-0.1,0.3) {$y_\ell$};
            \node at (-0.1,-0.5) {$\vartheta_\ell- \frac{\vartheta_{\ell+1}}{\vartheta'_{\ell+1}}\vartheta'_\ell$};
                  \draw(2.1,0.1)--(2.1,-0.1);
           \node  at (2.1,0.3) {$y_{\ell+1}$};
            \node at (2.1,-0.3) {$0$};

         { \color{red}{   \draw(-0.1,0)--(1,2);}
            \node  at (-0.2,1.1) {$\frac{\vartheta_{\ell+1}}{\vartheta'_{\ell+1}}\vartheta'_\ell$};
            \draw(2.1,0)--(1,2);
             \node  at (2.2,1.1) {$\vartheta_{\ell+1}$};}
           \node  at (2,1.7) {$x_2$};
          \draw(5,2.1)--(5,1.9);
           \node  at (5,1.7) {$x_N$};
                  \draw(6,0.1)--(6,-0.1);
           \node  at (6,0.3) {$y_M$};

        \end{tikzpicture}
        \caption{Case II, Case 3. The figure shows the new probability masses of the involved atoms}
  \end{center}
    \end{figure}

   \item[Case 4] $\vartheta'_\ell > \vartheta_\ell, \vartheta'_{\ell+1}> \vartheta_{\ell+1}.$

Define here $k := \min\big\{\frac{\vartheta_{\ell}}{\vartheta'_{\ell}}, \frac{\vartheta_{\ell+1}}{\vartheta'_{\ell+1}} \big\}.$ If $k= \frac{\vartheta_{\ell}}{\vartheta'_{\ell}}$ proceed as in case 2. If $k=\frac{\vartheta_{\ell+1}}{\vartheta'_{\ell+1}}$ proceed as in case 3.
\end{description}
\end{description}

\begin{theorem}
The algorithm produces the left-monotone transport plan in a finite number of steps.
\end{theorem}

\begin{proof}
In each step of the algorithm at least one atom is crossed from the list. In the final step at least two atoms vanish. Thus, the algorithm stops in at most $N+M-1$ steps. Moreover, the algorithms produces a $\Q\in  \mathcal{M}_{2}(\mu,\nu)$  since
\begin{enumerate}
\item the whole probability mass $\omega_j$ is transported from $x_j$ and the whole mass $\vartheta_i$ is received by $y_i$.
\item the martingale property is satisfied. This is obvious in Case I and guaranteed by the second linear equations  in Case II.
\end{enumerate}
What is left to show is that the algorithm continues after each step, i.e. that the remaining probability masses still satisfy $\mu\le_c \nu.$ For this property indeed it is crucial that the algorithm starts with the smallest atom $x_1$. We will show this statement for each case separately by using the characterization of Lemma \ref{Prop Char Conv Ord MultiMarginal} of the convex order. Note that the condition $\frac{1}{\mu(\R)}\int x\mu(dx)= \frac{1}{\nu(\R)}\int x\nu(dx)$ is satisfied in each step, since the same probability masses are subtracted on both sides and the integrals remain the same due to the second equation.

\begin{description}
\item[Case I]  There exists $\ell \in \{1,\ldots,M\}$ such that $x_1=y_\ell$.
Initially we have by assumption that \begin{equation}\label{eq:cnew}
\Eop[(X-t)_+]= \sum_{j=1}^N(x_j-t)_+\omega_j \le\sum_{i=1}^M(y_i-t)_+\vartheta_i= \Eop[(Y-t)_+].\end{equation} In Case I the new probability mass is constructed by subtracting on both sides the same mass which belongs to the same value. Thus, obviously the inequality still holds true and the remaining masses still satisfy $\mu\le_c\nu.$
\item[Case II]  There exists $\ell \in \{1,\ldots,M-1\}$ such that $y_\ell<x_1<y_{\ell+1}$.
\begin{description}
 \item[Case 1] $\vartheta'_\ell \le \vartheta_\ell, \vartheta'_{\ell+1}\le \vartheta_{\ell+1}.$

 First assume that $t\le y_\ell$. In what follows we will always set $m_t := \min\{i: y_i > t\}$. By assumption we have that
 $$\Eop[(X-t)_+]= \sum_{j=1}^N(x_j-t)\omega_j \le\sum_{i=m_t}^M(y_i-t)\vartheta_i= \Eop[(Y-t)_+].$$ This implies
 \begin{eqnarray*}
 \sum_{j=2}^N(x_j-t)\omega_j  &\le & \sum_{i=m_t}^M(y_i-t)\vartheta_i-(x_1-t) \omega_1\\
 &=& \sum_{i=m_t}^M(y_i-t)\vartheta_i-(y_\ell-t)\vartheta'_\ell -(y_{\ell+1}-t)\vartheta'_{\ell+1}
 \end{eqnarray*}
  which implies that the new masses still satisfy the inequality.
  Now assume that $t\ge y_{\ell+1}$. In this case all atoms where masses are changed do not enter the expectation. Hence obviously the inequality \eqref{eq:cnew} is satisfied for the new masses. Finally assume that $t\in (y_\ell,y_{\ell+1})$ and let $n_t := \min\{j: x_j > t\}$. Here we have to show that
  $$  \sum_{j=n_t}^N(x_j-t)\omega_j  \le \sum_{i=\ell+1}^M(y_i-t)\vartheta_i-(y_{\ell+1}-t)\vartheta'_{\ell+1}$$
which is equivalent to showing that $f(t)\ge 0$ for $ t\in (y_\ell,y_{\ell+1})$ where
$$ f(t) := \sum_{i=\ell+1}^My_i\vartheta_i-y_{\ell+1}\vartheta'_{\ell+1}-
\sum_{j=n_t}^Nx_j\omega_j + t\Big( \sum_{j=n_t}^N \omega_j - \sum_{i=\ell+1}^M \vartheta_i+\vartheta'_{\ell+1}\Big).
$$
Note that
\begin{enumerate}
\item $f$ is continuous and piecewise linear (it changes its slope at $x_j$).
\item $f(y_\ell)=f(y_{\ell+1})\ge 0.$
\item the slope is decreasing in $t$, hence $f$ is concave.
\end{enumerate}
Thus we can conclude that $f(t)\ge 0$ for $ t\in (y_\ell,y_{\ell+1})$.
  \item[Case 2] $\vartheta'_\ell > \vartheta_\ell, \vartheta'_{\ell+1}\le \vartheta_{\ell+1}.$

   First assume that $t\le y_\ell$. By assumption we have that
 $$\Eop[(X-t)_+]= \sum_{j=1}^N(x_j-t)\omega_j \le\sum_{i=m_t}^M(y_i-t)\vartheta_i= \Eop[(Y-t)_+].$$ This implies
 \begin{eqnarray*}
 \sum_{j=1}^N(x_j-t)\omega_j  -(x_1-t)\omega_1 \frac{\vartheta_\ell}{\vartheta'_{\ell}}&\le & \sum_{i=m_t}^M(y_i-t)\vartheta_i-(x_1-t)\omega_1 \frac{\vartheta_\ell}{\vartheta'_{\ell}}\\ &=&\sum_{i=m_t}^M(y_i-t)\vartheta_i-\\ &&-(y_\ell-t)\vartheta_\ell-(y_{\ell+1}-t)  \frac{\vartheta_\ell}{\vartheta'_{\ell}} \vartheta'_{\ell+1}
 \end{eqnarray*}
  which implies that the new masses still satisfy the inequality.   Now assume that $t\ge y_{\ell+1}$. In this case all atoms where masses are changed do not enter the expectation. Hence obviously the inequality \eqref{eq:cnew} is satisfied for the new masses. Finally assume that $t\in (y_\ell,y_{\ell+1})$. Define
 \begin{eqnarray*}f(t) &:=& \sum_{i=\ell+1}^My_i\vartheta_i-y_{\ell+1}\frac{\vartheta_\ell}{\vartheta'_\ell}\vartheta'_{\ell+1}-
\sum_{j=n_t}^Nx_j\omega_j -1_{[n_t=1]}x_1\omega_1\frac{\vartheta_\ell}{\vartheta'_\ell}\\&&+ t\Big( \sum_{j=n_t}^N \omega_j - 1_{[n_t=1]}\omega_1\frac{\vartheta_\ell}{\vartheta'_\ell}-\sum_{i=\ell+1}^M \vartheta_i+\frac{\vartheta_\ell}{\vartheta'_\ell}\vartheta'_{\ell+1}\Big).
\end{eqnarray*}
Note that
\begin{enumerate}
\item $f$ is continuous and piecewise linear (it changes its slope at $x_j$).
\item $f(y_\ell)=f(y_{\ell+1})\ge 0.$
\item the slope is decreasing in $t$, hence $f$ is concave.
\end{enumerate}
  Thus we can conclude that $f(t)\ge 0$ for $ t\in (y_\ell,y_{\ell+1})$.
    \item[Case 3] $\vartheta'_\ell \le \vartheta_\ell, \vartheta'_{\ell+1}> \vartheta_{\ell+1}.$

      First assume that $t\le y_\ell$. By assumption we have that
 $$\Eop[(X-t)_+]= \sum_{j=1}^N(x_j-t)\omega_j \le\sum_{i=m_t}^M(y_i-t)\vartheta_i= \Eop[(Y-t)_+].$$ This implies
 \begin{eqnarray*}
 \sum_{j=1}^N(x_j-t)\omega_j  -(x_1-t)\omega_1 \frac{\vartheta_{\ell+1}}{\vartheta'_{\ell+1}}&\le & \sum_{i=m_t}^M(y_i-t)\vartheta_i-(x_1-t)\omega_1 \frac{\vartheta_{\ell+1}}{\vartheta'_{\ell+1}}\\ &=&\sum_{i=m_t}^M(y_i-t)\vartheta_i-\\ &&-(y_\ell-t)\vartheta'_\ell \frac{\vartheta_{\ell+1}}{\vartheta'_{\ell+1}}     -(y_{\ell+1}-t)  \vartheta_{\ell+1}
 \end{eqnarray*}
  which implies that the new masses still satisfy the inequality.  Now assume that $t\ge y_{\ell+1}$. In this case all atoms where masses are changed do not enter the expectation. Hence obviously the inequality \eqref{eq:cnew} is satisfied for the new masses. Finally assume that $t\in (y_\ell,y_{\ell+1})$. Define
 \begin{eqnarray*}f(t) &:=& \sum_{i=\ell+2}^My_i\vartheta_i-
\sum_{j=n_t}^Nx_j\omega_j -1_{[n_t=1]}x_1\omega_1\frac{\vartheta_{\ell+1}}{\vartheta'_{\ell+1}}\\&&+ t\Big( \sum_{j=n_t}^N \omega_j - 1_{[n_t=1]}\omega_1\frac{\vartheta_{\ell+1}}{\vartheta'_{\ell+1}}-\sum_{i=\ell+2}^M \vartheta_i\Big).
\end{eqnarray*}
  Note that
\begin{enumerate}
\item $f$ is continuous and piecewise linear (it changes its slope at $x_j$).
\item $f(y_\ell)=f(y_{\ell+1})\ge 0.$
\item the slope is decreasing in $t$, hence $f$ is concave.
\end{enumerate}
  Thus we can conclude that $f(t)\ge 0$ for $ t\in (y_\ell,y_{\ell+1})$.

       \item[Case 4] $\vartheta'_\ell > \vartheta_\ell, \vartheta'_{\ell+1}> \vartheta_{\ell+1}.$

       This case either leads to case 2 or case 3.
       \end{description}
\end{description}
In total we have shown that the algorithm can be continued the same way in the next step.
\end{proof}

\section{An Algorithm for the Construction of the Optimal Solution of (D)}\label{sec:dual}
In order to construct an optimal solution for the dual problem it is clear due to the complementary slackness condition that whenever $q_{j,i}>0$ we must have
$$ \varphi_j + \psi_i + h_j (y_i-x_j)= c_{j,i}.$$
Thus, when performing the algorithm for the primal problem we can at the same time fix some variables for the optimal solution of the hedging problem. This algorithm is again recursive. Note that we have to replace $1$ and $\ell$ by the true indices of these atoms.

{\bf Algorithm (dual problem):}

\begin{description}
\item[Case I]  There exists $\ell \in \{1,\ldots,M\}$ such that $x_1=y_\ell$.
\begin{description}
\item[Case 1] $\omega_1\le \vartheta_\ell$.

In this case $x_1$ is crossed from the list and we fix $\varphi_1 := c_{1,\ell}-\psi_\ell$ since it will not appear in the remaining equations.
 \item[Case 2] $\omega_1> \vartheta_\ell$.
 In this case $y_\ell$ is crossed from the list and we fix $\psi_\ell := c_{1,\ell}-\varphi_1$ since it will not appear in the remaining equations.
 \end{description}

\item[Case II]  There exists $\ell \in \{1,\ldots,M-1\}$ such that $y_\ell<x_1<y_{\ell+1}$.
\begin{description}
 \item[Case 1] $\vartheta'_\ell \le \vartheta_\ell, \vartheta'_{\ell+1}\le \vartheta_{\ell+1}.$

 In this case we cross $x_1$ from the list and we fix $\varphi_1, h_1$ as the solution of
 \begin{eqnarray*}
 && \varphi_1 + \psi_\ell + h_1 (y_\ell-x_1)= c_{1,\ell}\\
 && \varphi_1 + \psi_{\ell+1} + h_1 (y_{\ell+1}-x_1)= c_{1,\ell+1}
 \end{eqnarray*}
 since this atom will not appear in the remaining equations.

 \item[Case 2] $\vartheta'_\ell > \vartheta_\ell, \vartheta'_{\ell+1}\le \vartheta_{\ell+1}.$

 In this case we cross $y_\ell$ from the list and we fix $\psi_\ell$ as
 \begin{eqnarray*}
 &&  \psi_\ell = c_{1,\ell} - \varphi_1 - h_1 (y_\ell-x_1)
 \end{eqnarray*}
 since this atom will not appear in the remaining equations.

    \item[Case 3] $\vartheta'_\ell \le \vartheta_\ell, \vartheta'_{\ell+1}> \vartheta_{\ell+1}.$

  In this case we cross $y_{\ell+1}$ from the list and we fix $\psi_{\ell+1}$ as
 \begin{eqnarray*}
 &&  \psi_{\ell+1} = c_{1,\ell+1} - \varphi_1 - h_1 (y_{\ell+1}-x_1)
 \end{eqnarray*}
 since this atom will not appear in the remaining equations.

       \item[Case 4] $\vartheta'_\ell > \vartheta_\ell, \vartheta'_{\ell+1}> \vartheta_{\ell+1}.$

       This case either leads to case 2 or case 3.
       \end{description}
\end{description}

Note that in each step we can choose the free variable since we only fix variables when the corresponding atom is crossed from the list. Thus, during the algorithm all $\varphi_j$ are determined and some of the $h_j$ and $\psi_i$. The remaining $h_j$ and $\psi_i$ have to be chosen such that the inequalities
 $$ \varphi_j + \psi_i + h_j (y_i-x_j)\geq c_{j,i}$$
are satisfied and the objective function is minimized. This is possible due to complementary slackness.

\section{An Example}\label{sec:example}

In this section we illustrate our algorithms to solve the problems \ref{Prob Upper Price Bound Discrete} and \eqref{Prob Super Hedging Discrete} for a payoff function $c:\R^2 \to \R$ such that \eqref{eq:optcondition} is satisfied. We consider the margins
\[
  \mu=\frac{1}{2}(\delta_{1}+\delta_3) \leq_c \nu= \frac{1}{2}\delta_0 + \frac{1}{6}\delta_2 + \frac{1}{3}\delta_5
\]
and begin illustrating the algorithm that determines $\Q_l(\mu,\nu)$ and thus solves \eqref{Prob Upper Price Bound Discrete}. First, we couple the mass of $\delta_1$. Therefore we solve
\begin{eqnarray}
\nonumber \vartheta'_1 + \vartheta'_{2} &=& \omega_1=\frac12\\ \nonumber
 \vartheta'_1 \cdot 0 +\vartheta'_{2} \cdot 2 &=& \omega_1 \cdot 1 = \frac12,
\end{eqnarray}
which yields $\vartheta'_1=\vartheta'_2=\frac{1}{4}.$ As $\vartheta'_1 \le \vartheta_1, \vartheta'_2> \vartheta_2$ we are in \textbf{case II, 3}. Thus we define $q_{1,1}=q_{1,2}=\frac{1}{6}$, $\omega_1=\frac{1}{6}$, $\vartheta_1=\frac{1}{3}$ and $\vartheta_2=0$.

Then we proceed to couple the remaining mass of $\delta_1$ and thus solve
\begin{eqnarray}
\nonumber \vartheta'_1 + \vartheta'_{3} &=& \omega_1=\frac16\\ \nonumber
 \vartheta'_1 \cdot 0 +\vartheta'_{3} \cdot 5 &=& \omega_1 \cdot 1=\frac16,
\end{eqnarray}
which yields $\vartheta'_1= \frac{2}{15},\vartheta'_3=\frac{1}{30}.$ As $\vartheta'_1 \le \vartheta_1, \vartheta'_3\le \vartheta_3$ we are in \textbf{case  II, 1}. Thus we define $q_{1,1}=\frac{2}{15}, q_{1,3}=\frac{1}{30}$, $\omega_1=0$, $\vartheta_1=\frac{1}{5}$ and $\vartheta_3=\frac{3}{10}$.

Finally, we couple the mass of $\delta_3$. Therefore we solve
\begin{eqnarray}
\nonumber \vartheta'_1 + \vartheta'_{3} &=& \omega_2=\frac12\\ \nonumber
 \vartheta'_1 \cdot 0 +\vartheta'_{3} \cdot 5 &=& \omega_2 \cdot 3=\frac32,
\end{eqnarray}
which clearly yields $\vartheta'_1= \frac{1}{5},\vartheta'_3=\frac{3}{10}$ and as $\vartheta'_1 \le \vartheta_1, \vartheta'_3\le \vartheta_3$ we are in \textbf{case II, 1}. Thus we define $q_{2,1}=\frac{1}{5}, q_{2,3}=\frac{3}{10}$, $\omega_2=0$, $\vartheta_1=0$ and $\vartheta_3=0$. Note that in this example mass between the same atoms is transported in two different steps. These masses have to be added in the end.

In total this yields the following unique left-monotone transport plan \[
  \Q_l(\mu,\nu)=\frac{3}{10}\delta_{(1,0)} + \frac{1}{6}\delta_{(1,2)} + \frac{1}{30}\delta_{(1,5)} + \frac{1}{5}\delta_{(3,0)}+ \frac{3}{10}\delta_{(3,5)}.
\]

    \begin{figure}[!htbp]\begin{center}

       \begin{tikzpicture}[scale=1]

          \node  at (0,0.4) {$1$};
          \node at (2,0.4) {$3$};
         \fill (0,0) circle (0.05);
         \fill (2,0) circle (0.05);

          \node at (-2,-3.4) {$0$};
      \node at (1,-3.4) {$2$};
          \node at (3,-3.4) {$5$};
           \fill (-2,-3) circle (0.05);
                  \fill (1,-3) circle (0.05);
                  \fill (3,-3) circle (0.05);

                  \draw(0,0)--(-2,-3);
                     \draw(0,0)--(1,-3);
                      \draw(0,0)--(3,-3);
                     \draw(2,0)--(-2,-3);
                      \draw(2,0)--(3,-3);

                        \node  at (-1.4,-1.5) {$\frac{3}{10}$};
                         \node  at (2.8,-1.5) {$\frac{3}{10}$};
                          \node  at (1.2,-2.5) {$\frac{1}{6}$};
                            \node  at (-0.9,-2.5) {$\frac{1}{5}$};
                              \node  at (1.2,-1.5) {$\frac{1}{30}$};

        \end{tikzpicture}
        \caption{Left-monotone transport plan in the example.}
  \end{center}
    \end{figure}
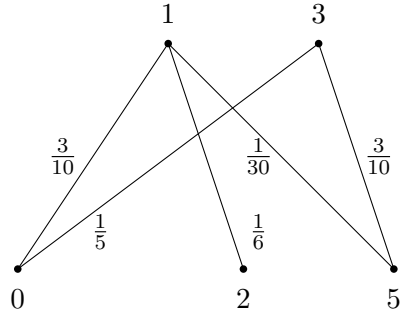

Parallel to determining $\Q_{l}(\mu,\nu)$ we may determine the hedging functions $\varphi_1,\varphi_2$, $\psi_2$, and $h_1,h_2$ as follows. In the first step, where we are in \textbf{case II, 3}, we define \[
  \psi_2=c_{1,2}-\varphi_1-h_1(y_2-x_1)=c_{1,2}-\varphi_1-h_1.
\]
In the second step, where we are in \textbf{case II,1}, we solve the given linear equation system which yields\[
  h_1=\frac{c_{1,3}-c_{1,1}-(\psi_3-\psi_1)}{y_3-y_1}=\frac{c_{1,3}-c_{1,1}-(\psi_3-\psi_1)}{5}
\]
and
\[
  \varphi_1 = \frac{y_3-x_1}{y_3-y_1}(c_{1,1}-\psi_1)+\frac{x_1-y_1}{y_3-y_1}(c_{1,3}-\psi_3) = \frac{4}{5}(c_{1,1}-\psi_1)+\frac{1}{5}(c_{1,3}-\psi_3).
\]
Finally, in the third step we analogously determine
\[
  h_2=\frac{c_{2,3}-c_{2,1}-(\psi_3-\psi_1)}{y_3-y_1}=\frac{c_{2,3}-c_{2,1}-(\psi_3-\psi_1)}{5}
\]
and
\[
  \varphi_2 = \frac{y_3-x_2}{y_3-y_1}(c_{2,1}-\psi_1)+\frac{x_2-y_1}{y_3-y_1}(c_{2,3}-\psi_3) = \frac{2}{5}(c_{2,1}-\psi_1)+\frac{3}{5}(c_{2,3}-\psi_3).
\]
Then by construction we have \[
  \varphi_j+\psi_i+h_j(y_i-x_j)=c_{j,i}
\]
for all $j,i$ such that $(x_j,y_i) \in \mathrm{supp}(\Q_{l}(\mu,\nu)).$ The only remaining inequality to be satisfied is
\[
  \varphi_2+\psi_2+h_2(y_2-x_2)\geq c_{2,2}.
\]
Plugging in all defined notions, after a short calculation we obtain that the above is equivalent to\[
  \frac{3}{5}(c_{2,1}-c_{1,1})+\frac{2}{5}(c_{2,3}-c_{1,3})-(c_{2,2}-c_{1,2}) \geq 0,
\] which is satisfied by assumption. That is, we may define $\psi_1=\psi_3=0$ and thus obtain
$
  \varphi_1=\frac{4}{5}c_{1,1}+\frac{1}{5}c_{1,3},\ h_1=\frac{c_{1,3}-c_{1,1}}{5},\ \varphi_2=\frac{2}{5}c_{2,1}+\frac{3}{5}c_{2,3},\ h_2=\frac{c_{2,3}-c_{2,1}}{5},$ and $\psi_{2}=c_{1,2}-\frac{3}{5}c_{1,1}-\frac{2}{5}c_{1,3}.
$
Note that the solution of the dual problem is not unique.
For further examples we refer the reader to \cite{schm18}.

\end{document}